\documentclass[a4paper,10pt]{article}
\usepackage[utf8]{inputenc}
\usepackage{amsmath,amsthm,amsfonts,amssymb}
\usepackage{enumerate}

\newcommand{\scrA}{\mathcal{A}}

\newcommand{\scrC}{\mathcal{C}}

\newcommand{\scrH}{\mathcal{H}}

\newcommand{\scrP}{\mathcal{P}}

\newcommand{\scrR}{\mathcal{R}}

\newcommand{\scrX}{\mathcal{X}}
\newcommand{\scrY}{\mathcal{Y}}
\newcommand{\scrZ}{\mathcal{Z}}

\newcommand{\im}{\mathrm{Im}}
\newcommand{\eps}{\varepsilon}
\newcommand{\gauss}[2]{\genfrac{[}{]}{0pt}{}{#1}{#2}}

\newcommand{\CC}{\mathbb{C}}

\newcommand{\RR}{\mathbb{R}}
\newcommand{\I}{\mathbf{I}}
\newcommand{\rank}{\mathrm{rank}}
\newcommand{\disc}{\mathrm{disc}}
\newcommand{\codim}{\mathrm{codim}}
\newcommand{\PG}{\mathrm{PG}}
\newcommand{\AG}{\mathrm{AG}}
\newcommand{\GF}{\mathrm{GF}}
\newcommand{\<}{\langle}
\renewcommand{\>}{\rangle} 

\newtheorem{theorem}{Theorem}[section]
\newtheorem{lemma}[theorem]{Lemma}

\theoremstyle{remark}

\theoremstyle{definition}

\newtheorem{example}[theorem]{Example}

\usepackage{url}

\title{A Survey of Cameron-Liebler Sets and Low Degree Boolean Functions in Grassmann Graphs}
\author{Ferdinand Ihringer}

\begin{document}

\maketitle

\begin{abstract}
We survey results for Cameron-Liebler sets and
low degree Boolean functions for Hamming graphs,
Johnson graphs and Grassmann graphs
from the point of view of association schemes.
This survey covers selected results in finite geometry, Boolean function analysis,
design theory, coding theory, and cryptography.
\end{abstract}



\section{Introduction}

We survey \textit{Cameron-Liebler sets} and \textit{low degree Boolean functions}
in hypercube, Johnson graph and Grassmann graph.
Going back to work by Cameron and Liebler in 1982 on permutation groups \cite{CL1982},
the study of Cameron-Liebler sets and \textit{antidesigns}
in finite geometries has been a particularly flourishing topic in recent years.
The seemingly unrelated \textit{Fourier analytic study} of Boolean functions on the hypercube $\{0, 1\}^n$
is a fundamental topic in theoretical computer science which emerged in the 1990s.
Recently, Boolean function analysis expanded from the study of the hypercube
towards the investigation of more algebraic structures such as the subspace lattice of a finite vector space
or bilinear forms over finite fields. Most notably, the remarkable proof of the $2$-to-$2$ Games Conjecture
by Khot, Muli and Safra \cite{KMS2023} has been obtained by the study of
expansion properties of Grassmann graphs.
Thus, both fields converge, but due to different traditions the exchange
of ideas remains low. One reason for this is the lack of mutual intelligibility of publications
in all the involved areas. We hope that this survey helps with amending this situation.\footnote{%
And if not, then at least it summarizes all the aspects of the area which the author enjoys.}

Our aim is a survey of selected results written from a perspective of association schemes
and Delsarte theory. This provides a reasonably uniform framework which
most researchers in finite geometry, spectral graph theory, finite group theory, coding theory
and design theory are familiar with. For readers coming from Boolean function analysis
and closely related areas in theoretical computer science,
the advantage is that most things are phrased in terms of linear algebra and
we can mostly avoid the vast specific terminology of areas such as finite geometry
or Delsarte-style coding theory.
We will also mention some parallel developments in cryptography
to clarify some precedence.

  This survey focusses on three structures:
  the hypercube $\{0, 1\}^n$,
  the Johnson graph of $m$-sets of $\{ 1, \ldots, n \}$
  (or slice of the hypercube), and the Grassmann graph of $m$-spaces in $V(n, q)$
  (where $V(n, q)$ denotes the $n$-dimensional vector space over $\GF(q)$, the field with $q$ elements).
  At the end we will also mention some other recent developments
  in permutation groups as well as other finite geometries such as bilinear forms
  and affine spaces.

\section{Preliminaries}

The main difficulty of the surveyed subject is that it spans several communities,
each of them using their own notation. Thus, we start by surveying several
perspectives on the topic.


\subsection{Spectral Graph Theory}

A graph $\Gamma = (X, \sim)$ has vertex set $X$ and an adjacency relation $\sim$.
For us, all graphs are finite and simple.

Let $v=|X|$. Say that $\Gamma$ is $k$-regular if each vertex is adjacent to exactly $k$ vertices.
Denote the {\it adjacency matrix} of $\Gamma$ by $A$.
Write $\lambda_1 \geq \lambda_2 \geq \ldots \geq \lambda_v$ for the spectrum of a symmetric matrix $A$.
Let $u_1, \ldots, u_v$ be a corresponding orthogonal basis of eigenvectors.
Let $V(\lambda)$ denote the eigenspace of $\lambda$.
We denote the idempotent projection onto $V(\lambda)$ by $E(\lambda)$,
so $E(\lambda) = \sum_{\lambda_j = \lambda} u_j u_j^T$.
Clearly, $\rank(E(\lambda)) = \dim(V(\lambda))$ and
\[
 A = \sum_{\lambda \in \{ \lambda_1, \ldots, \lambda_v \}} \lambda E(\lambda).
\]

For a subset $Y \subseteq X$, we write $f_Y$ for the \textit{characteristic function}
(or: \textit{characteristic vector}) of $Y$.
That is, $f_Y(S) = 1$ if $S \in Y$ and $f_Y(S) = 0$ otherwise.
There is some easy, but useful arithmetic. Consider a subspace $W$ of $\CC^v$ with the all-ones vector
$j$ in $W$.
\begin{enumerate}[(i)]
 \item If $f_Y \in W$, then $f_{X \setminus Y} \in W$.
 \item If $Y$ and $Z$ disjoint and $f_Y, f_Z \in W$, then $f_{Y \cup Z} = f_Y + f_Z \in W$.
 \item If $Z \subseteq Y$ and $f_Y, f_Z \in W$, then $f_{Y \setminus Z} = f_Y - f_Z \in W$.
\end{enumerate}
In particular, if one intends to classify sets $Y$ with $f_Y \in W$ (as we want to),
then it suffices to execute this classification up to these three operations.

We will survey the following graphs.
\begin{example}[Hamming Graph]
  Let $X = \{ 0, \ldots, q-1 \}^n$. Say that $x \sim y$ if $x$ and $y$ have Hamming distance $1$.
  Then $(X, \sim)$ is the Hamming graph $H(n, q)$. The case $q=2$ is known as the {\it hypercube}.
  It has $n+1$ distinct eigenvalues $\theta_j = q(n-j)-d$ with $\dim(V(\theta_j)) = \binom{d}{j} (q-1)^j$,
  where $0 \leq j \leq n$, see \cite[Th.~9.2.1]{BCN}.
\end{example}

\begin{example}[Johnson Graph]
  Write $[n]$ for $\{ 1, \ldots, n \}$ and $\binom{[n]}{m}$ for all $m$-sets in $[n]$.
  Let $X = \binom{[n]}{m}$. Say that $x \sim y$ if $|x \setminus y| = 1$.
  Then $(X, \sim)$ is the \textit{Johnson graph} $J(n, m)$ or the \textit{slice (of the hypercube)}.
  For $n \geq 2m$, it has $m+1$ distinct eigenvalues $\theta_j = (m-j)(n-m-j)-j$
  with $\dim(V(\theta_j)) = \binom{n}{j} - \binom{n}{j-1}$,
  where $0 \leq j \leq n$, see \cite[Th.~9.1.2]{BCN}.
\end{example}

\begin{example}[Grassmann Graph]
 For $q$ a prime power, put $V = V(n, q)$.
 Write $\gauss{V}{m}$ for the set of all $m$-spaces in $V$ and define the
 Gaussian coefficient (or: $q$-binomial coefficient) by
 $\gauss{n}{m}_q = \gauss{n}{m} = \left| \gauss{V}{m} \right|$.
 Note that
 \[
  \gauss{n}{m} = \frac{(q^n-1)\cdots (q^{n-m+1}-1)}{(q^m-1) \cdots (q-1)}.
 \]
 Let $X = \binom{V}{m}$. Say that $x \sim y$ if $\dim(x \cap y) = m-1$.
 Then $(X, \sim)$ is the \textit{Grassmann graph} $J_q(n, m)$ or \textit{$q$-Johnson graph}.
  For $n \geq 2m$, it has $m+1$ distinct eigenvalues $\theta_j = q^{j+1} \gauss{m-j}{1}\gauss{n-m-j}{1} - \gauss{j}{1}$
  with $\dim(V(\theta_j)) = \binom{n}{j} - \binom{n}{j-1}$,
  where $0 \leq j \leq n$, see \cite[Th.~9.1.2]{BCN}.
\end{example}

\subsection{Association Schemes}

Let $X$ be a finite set. A symmetric $m$-class association scheme
is a pair $(X, \scrR)$ such that
\begin{enumerate}[(i)]
 \item $\scrR = \{ R_0, R_1, \ldots, R_m\}$ is a partition of $X \times X$;
 \item $R_0 = \{ (x,x): x \in X \}$;
 \item for each $i$, $0 \leq i \leq m$, $R_i = R_i^T$.
 \item there are numbers $p^k_{ij}$ such that for any $(x, y) \in R_k$ the number of $z \in X$
 with $(x, z) \in R_i$ and $(z, y) \in R_j$ is $p^k_{ij}$.
\end{enumerate}
Note that (iii) implies $p^k_{ij} = p^k_{ji}$.
Delsarte's PhD thesis is a good introduction to the topic \cite{Delsarte1973}.

Let $A_i$ denote the adjacency matrix of the graph $\Gamma_i = (X, R_i)$. 
Observe that $\Gamma_i$ is regular of degree $k_i := p^0_{ii}$.
Let $J$ denote the all-ones matrix and $I$ the identity matrix.
The properties (i)--(iv) above translate to
\begin{align*}
 &\sum_{i=0}^m A_i = J, && A_0 = I, && A_i = A_i^T, && A_iA_j = \sum_{k=0}^m p^k_{ij} A_k.
\end{align*}
As $p^k_{ij} = p^k_{ji}$, $A_iA_j = A_jA_i$. Thus, the $A_i$ generate
a $(m+1)$-dimensional commutative algebra $\scrA$. Hence, we can diagonalize the $A_i$
simultaneously (over $\CC$). We obtain a decomposition of $\CC^v$ into $m+1$
(maximal common) eigenspaces $V_0, V_1, \ldots, V_n$ of dimensions
$f_0, f_1, \ldots, f_m$. As $J \in \scrA$, one eigenspace has dimension $1$ and is spanned
by the all-ones vector ${\bf j}$. By convention, $V_0 = \<{\bf j}\>$ and $f_0 = 1$.
Let $E_j$ be the idempotent projection onto the $j$-th eigenspace $V_j$. This is a basis of
minimal idempotents of $\scrA$. We have
\begin{align*}
    & \sum_{j=0}^m E_j = I, && E_0 = v^{-1} J.\\
\intertext{There exist constants $P_{ji}$ and $Q_{ij}$ such that}
    & A_i = \sum_{j=0}^m P_{ji} E_j, && E_j = v^{-1} \sum_{i=1}^m Q_{ij} A_i.
\end{align*}
Note that the $P_{ji}$ are the eigenvalues of $A_i$.
A useful equality is
\begin{align*}
  f_i P_{ij} = k_j Q_{ji}. 
\end{align*}

For a non-empty subset $Y \subseteq X$ of size $y$ with characteristic vector $f_Y$,
define the {\it inner distribution} $a$ of $Y$ by
\begin{align*}
    a_i = \frac1y f_Y^T A_i f_Y.
\end{align*}
Delsarte's linear programming (LP) bound states that
\begin{align}
    (aQ)_j = \frac{v}{y} f_Y^T E_j f_Y \geq 0\label{eq:delsarte_lp}
\end{align}
for all $j \in \{ 0, 1, \ldots, m\}$, see Proposition 2.3.2 in \cite{BCN}.
The inequality follows from the fact that $E_j$ is positive semidefinite.
It is known as \emph{MacWilliams transform} of the vector $a$.
It implies that we can check if $f_Y$ is orthogonal to an eigenspace $V_j$
by looking at the $(m+1){\times}(m+1)$-matrix $Q$ (instead of the usually
much larger $(v\times v)$-matrix $E_j$).

Define $\| f \| = \left( \sum f_i^2 \right)^{1/2}$ as the $2$-norm for vectors,
not the (normalized) $2$-norm $\| f \|_2 = \left( \frac1v \sum f_i^2 \right)^{1/2}$ as used
in Boolean function analysis. We distinguish both by writing
$\| \cdot \|$ and $\| \cdot \|_2$, respectively.
Then
\begin{align}
  |Y| \cdot a_i = \sum_{j=0}^m P_{ji} \| E_j f_Y\|^2. \label{eq:projection}
\end{align}
Hence, we can investigate expansion (as in \textit{expander-mixing lemma} or \textit{expander graph})
for all graphs in an association scheme using the $(m+1)\times (m+1)$-matrices.

Two important classes of association schemes are {\it metric} and {\it cometric} association schemes.
An association scheme is metric if there exists an ordering of the $A_i$
and polynomials $p_i$ of degree $i$ such that
$A_i = p_i(A_1)$. An association scheme is cometric if there exists an
ordering of the $E_j$ and polynomials $q_j$ of degree $j$
such that $E_j = q_j(E_1)$ (here multiplication is $\circ$, the Hadamard product).
Metric association schemes are also known as {\it distance-regular graphs} or {\it $P$-polynomial
association schemes}. Cometric association schemes are also known as {\it $Q$-polynomial
association schemes}. All our three main examples belong to association schemes
which are metric and cometric.

\begin{example}[Hamming Scheme]
  Let $X = \{ 0, \ldots, q-1 \}^n$. Say that $(x, y) \in R_i$ if $x$ and $y$
  have Hamming distance $i$. This defines the $n$-class Hamming scheme, see \cite[\S9.2]{BCN}.
  Particularly, $(X, R_1)$ is the Hamming graph $H(n, q)$.
  Formulas for the $P_{ji}$'s are given by the Krawtchouk polynomials,
  see \cite[p.~39]{Delsarte1973}.
\end{example}

\begin{example}[Johnson Scheme]
  Let $n-m \geq m$.
  Let $X = \binom{[n]}{m}$. Say that $(x, y) \in R_i$ if $|x \setminus y| = i$.
  This defines the Johnson scheme $J(n, m)$, an $m$-class association scheme, see \cite[\S9.1]{BCN}.
  In particular, $(X, R_1)$ is the Johnson graph $J(n, m)$ and $(X, R_m)$ is the \textit{Kneser graph}.
  Formulas for the $P_{ji}$'s are given by the Eberlein polynomials, see Delsarte \cite[p.~48]{Delsarte1973}.
\end{example}

\begin{example}[Grassmann Scheme]
    Let $n \geq m$ and $V = V(n, q)$.
    Let $X = \binom{V}{m}$. Say that $(x, y) \in R_i$ if $\dim(x \cap y) = m-i$.
  This defines the Grassmann scheme $J_q(n, m)$ (also: $q$-Johnson scheme), an $m$-class association scheme, see \cite[\S9.3]{BCN}.
  In particular, $(X, R_1)$ is the Grassmann graph $J_q(n, m)$.
  Formulas for the $P_{ji}$'s are given by the $q$-Eberlein polynomials, see Delsarte \cite{Delsarte1976a}.
\end{example}

All formulas for the $P_{ji}$ for these examples which are known to the author can be found in \cite{BCIMG2018}.

Another useful way of determining if $f_Y \in \sum_{j \in J} V_j$ for some subset $J$ of $\{ 0, \ldots, m \}$
is \textit{design-orthogonality}.
Let $Y$ be a subset of $X$.
Suppose that $\scrZ$ is a family of subsets of $X$ such that $|Y \cap Z| = \frac{|Y|\cdot |Z|}{|X|}$
for all $Z \in \scrZ$. Note that $f_Y^T f_Z = |Y \cap Z|$.
Then for all $1 \leq j \leq m$, $E_j f_Y = 0$ or $E_j f_Z = 0$.
For example, take the Johnson scheme $J(n, m)$. There $X = \binom{[n]}{m}$.
Recall that a $d$-$(n, m, \lambda)$ design $Z$ is a subset of $X$ such that
any $d$-set of $[n]$ lies in precisely $\lambda$ elements of $Z$.
Suppose that $\scrZ$ is a family of $d$-$(n, m, \lambda)$ designs
and that $|Y \cap Z| = \frac{|Y|\cdot |Z|}{|X|}$ for all $Z \in \scrZ$.
Then $Y$ and $Z$ are design-orthogonal. In particular, $Y$ is an \textit{antidesign}.
Indeed, if a non-trivial $d$-$(n, m, \lambda)$ design $Z$ exists,
then for all images of $Z$ under the automorphism group,
the $f_Z$ span $V_0 + V_{d+1} + \cdots + V_m$ in the canonical
ordering of the eigenspaces of $J(n, m)$.
Hence, $f_Y \in V_0 + V_1 + \cdots + V_d$.
More generally, we can replace $f_Y$ and $f_Z$ by functions
$f \in V_0 + V_1 + \cdots + V_d$ and $g \in V_0 + V_{d+1} + \cdots + V_m$
and everything said remains true. In particular,
\begin{align}
 f^Tg = \frac{f^T{\bf j} \cdot g^T {\bf j}}{|X|}.\label{eq:designorth}
\end{align}
If $m$ divides $n$, then there exists a partition of $[n]$, respectively,
$V(n, q)$ into $m$-sets, respectively, $m$-spaces.
In the vector space case, this is a {\it spread}.
For a fixed $J(n, m)$ or $J_q(n, m)$, let $\scrZ$ denote the
set of all such partitions. Then the $f_Z$ with $Z \in \scrZ$
span $V_0 + V_2 + \cdots + V_m$. Cameron-Liebler sets are
often defined as those $Y$ such that $|Y \cap Z| = |Y| \cdot |Z|/|X|$
for all $Z \in \scrZ$. See \cite{DBSS2016} for $J(n, m)$
and \cite{BDBD2019} for $J_q(n, m)$.

Also see Ito \cite{Ito2004} for a treatment of the topic of design-orthogonality
from the point of view of \textit{coset geometries}.

\subsection{Graded Posets} 

Let $(\scrX, \subseteq)$ be a graded poset of rank $n$.
That is, a partially ordered set with
a rank function.
Let $X_i$ denote all elements of $\scrX$ of rank $i$.
Let $D$ be a subset of $\scrX$. For $D \in X_i$, define
$x_D\colon X \rightarrow \{ 0, 1 \}$ as the (Boolean) function with $x_D(S) = 1$ if
$D \subseteq S$ and $x_D(S) = 0$ otherwise.

Now only consider $x_P$ with $P \in X_1$.
Then for any $f: \scrX \rightarrow \CC$,
$f$ can be written as a multivariate polynomial
with $x_P$ as variables and degree at most $n$.
For instance,
\[
 \sum_{S \in \scrX} f(S) \left( \prod_{P \in X_1: P \subseteq S} x_P \right)
 \left( \prod_{P \in X_1: P \nsubseteq S} 1-x_P\right).
\]
We say that $f$ has degree $d$ if $d$ is the minimum degree of such a polynomial.
If $D \in \scrX_d$, then $x_D$ has degree $d$ as
\[
 x_D = \prod_{P \in X_1: P \nsubseteq D} x_P.
\]
If we consider $f: x_m \rightarrow \CC$, then $\deg(f) \leq m$.
This can be seen by considering
\[
 \sum_{S \in X_m} f(S) \prod_{P \in X_1: P \nsubseteq S} x_P.
\]

In several examples of metric association schemes, there is a metric ordering of relations
such that $(x, y) \in R_1$ if and only if $x$ and $y$ have minimal distance in $\scrX$ restricted to $X$
(usually, either distance $1$ if $X = \scrX$ or distance $2$ if $X$ is one layer of $\scrX$).
Similarly, there is often a cometric ordering such that
$f$ is a degree $d$ function if and only if $f \in V_0 + V_1 + \ldots + V_d$,
but not $f \in V_0 + V_1 + \ldots + V_{d-1}$. See also \cite{Delsarte1976a,Roos1982}.

In the literature on cometric association schemes, what we call \textit{degree}
is usually referred to as \textit{dual degree}. For instance, see \cite{BGKM2003}.
For sets $Y, Z \subseteq X$, let $f_Y$ and $f_Z$ be the corresponding Boolean functions,
that is, from now on we identify characteristic vectors and Boolean functions.

\begin{enumerate}[(i)]
 \item We have $\deg(f_Y) = \deg(1-f_Y)$.
 \item If $Y$ and $Z$ disjoint, then $f_Y+f_Z$ is Boolean with degree at most $\max(\allowbreak \deg(f_Y), \allowbreak \deg(f_Z))$.
 \item If $Z \subseteq Y$, then $f_Y - f_Z$ is Boolean with degree at most $\max(\deg(f_Y), \allowbreak \deg(f_Z))$.
\end{enumerate}

\begin{example}[Hamming Graph]
    There are (at least) two natural choices for a poset associated with
    the Hamming graph.
    The first one is very natural for $q=2$
    if we identify $\{ 0, 1 \}^n$ with the subset poset of $[n]$.
    Namely, $\scrX = \{ 0, \ldots, q-1 \}^n$ and we say that $x \subseteq y$
    if $x_i = y_i$ for all $i$ with $x_i \neq 0$.
    Here, for the Hamming scheme, $X = \scrX$.

    The second choice is that $\scrX$ is the set of all words of length $n$
    out of $\{ 0, \ldots, q-1 \} \cup \{ \cdot \}$ and we say that $x \subseteq y$
    if $x_i = y_i$ for all $i$ with $x_i \neq \cdot$. This is called the {\it Hamming lattice}.
    For the Hamming scheme, $X = X_n$.
\end{example}

\begin{example}[Johnson Graph]
  Let $\scrX$ be all subsets of $[n]$.
  Then $X=X_m$ for the Johnson scheme $J(n, m)$.
\end{example}

\begin{example}[Grassmann Graph]
    Let $\scrX$ be all subspaces of $V(n, q)$, the {\it subspace lattice} of $V(n, q)$.
    Then $X=X_m$ for the Grassmann scheme $J_q(n, m)$.
\end{example}

Suppose that $Y$ is a subset of vertices in any of the three examples.
Then $\deg(f_Y) \leq d$ if and only if
$f_Y \in V_0 + V_1 + \ldots + V_d$ in the usual cometric ordering.
Similarly, distance in the poset corresponds naturally to distance
in the graph.

Suppose that we are in the Johnson graph or the Grassmann graph.
If $\deg(f) \leq d$, then we can write $f$ as
\[
 f = \sum_{D \in X_d} c_D x_D
\]
for uniquely determined constants $c_D$. Hence, we only need
to consider $X_d \times X_m$ in this case.

\subsection{Equitable Partitions}

The concept of {\it equitable partition} is also related.
For instance, they often correspond to
the extremal examples in spectral bounds, see \cite{Haemers1995} and
our discussion in \S\ref{sec:ratio}.
Let $\Gamma$ be a graph of order $v$ with vertex set $X$. For $A,B$ sets of vertices of $\Gamma$, let $E(A, B)$
denote the set of edges of $\Gamma$ in $A \times B$.
Now let $\scrX = \{ X_1, X_2, \ldots, X_\nu \}$ be a partition of the vertex set $X$.
Let $e_{ij}$ be the average number of edges from $X_i$ to $X_j$, that is $e_{ij} = |E(X_i,X_j)|/|X_i|$.
The matrix $E = (e_{ij})$ of order $\nu$ is the \textit{quotient matrix} of the partition.
Let $\mu_1 \geq \mu_2 \geq \ldots \geq \mu_\nu$ be the eigenvalues of $E$.

\begin{theorem}
    We have $\lambda_j \geq \mu_j \geq \lambda_{v-\nu+j}$ for all $1 \leq j \leq \nu$.
\end{theorem}

If for all $i,j$ each vertex $x \in X_i$ has exactly $e_{ij}$ neighbors in $X_j$,
then the partition is an equitable partition.
Other words for {\it equitable partition} include \textit{perfect coloring},
\textit{regular partition}, and, for $\nu=2$, (often) \textit{regular set}
or (in parts of finite geometry) \textit{intriguing set}.
We have the following characterization of equality:

\begin{theorem}
    Suppose that $\scrX$ is equitable.
    Then there exist $1 \leq i_1 \leq \ldots \leq i_\nu \leq v$
    such that $\lambda_{i_j} = \mu_j$.
    Furthermore, $f_{X_i}$ lies in $V(\lambda_{i_1}) + \cdots + V(\lambda_{i_\nu})$.
\end{theorem}

The most interesting case for us is that of equitable bipartitions, that is, $\nu = 2$.
For example, take (the distance-$1$-graph on) the hypercube.
As $\mu_1$ is the degree for a regular graph, $\mu_1 = n$.
The eigenvalues of the hypercube are $n, n-2, \ldots, -n+2, -n$.
If $\mu_2 \geq n-2d$, then $f_{X_i}$ has degree at most $d$.

Special cases of equitable partitions often get rediscovered
and are investigated from various points of views,
so many names for relatively similar things exist.
For instance, finite geometry alone uses {\it Cameron-Liebler set}, {\it tight set},
{\it $m$-ovoid}, {\it hemisystem} and {\it intriguing set} for particular equitable partitions.
Design theory has {\it design} and {\it antidesign} \cite{Roos1982}. Coding theory
has {\it completely regular codes}, see \cite{Delsarte1973,Martin1994,GZ2025}.
In our three families, Boolean degree $1$ functions correspond to
{\it completely regular strength $0$ codes of covering radius $1$}
(but there is no such correspondence for larger degree).
Latin squares have {\it $k$-plexes}.
A more recent example is the notion of \textit{graphical design}
and \textit{extremal graphical design}, for instance,
see \cite{Golubev2020,Steinerberger2020,Zhu2023}.
Finally, let us mention that there is vast literature on equitable partitions themselves.
For instance, equitable partitions of graphs of degree
at most $5$ are classified \cite{DF2021}.

\subsection{Tactical Decompositions}

Let $X$ and $Y$ be finite sets and let $\I \in X \times Y$.
Call the elements of $X$ {\it points}, the elements of $Y$ {\it blocks},
and $\I$ an {\it incidence relation}. If $(x, y) \in \I$, then
$x$ and $y$ are {\it incident}. We write $x \,\I\, y$.

Let $\scrX = \{ X_1, X_2, \ldots, X_s \}$ be a partition of $X$
and let $\scrY = \{ Y_1, Y_2, \ldots, Y_t\}$ be a partition of $Y$.
Then $(\scrX, \scrY)$ is a \textit{tactical decomposition} (or: {\it generalized orbit})
if the following holds:
\begin{enumerate}
 \item For $x \in X_i$, the number of $y \in Y_j$ with $x \,\I\, y$ only depends on $j$.
 \item For $y \in Y_i$, the number of $x \in X_j$ with $x \,\I\, y$ only depends on $i$.
\end{enumerate}
Define a (bipartite) graph $\Gamma = (X \cup Y, \sim)$ where $x \sim y$ if $x \,\I\, y$.
Hence, a tactical decomposition is an equitable partition of a bipartite incidence graph.

Many examples for tactical decomposition come from finite groups.
In this context Block's lemma \cite{Block1967} is important as it shows
(under certain assumptions) for the group case that $s \leq t$.
We use the formulation given by Vanhove \cite{FV}.

\begin{lemma}[Block's Lemma]\label{lem:block}
Let $G$ be a group acting on two finite sets $X$ and $Y$, with respective sizes $n$ and $m$.
Let $O_1,\ldots,O_s$, respectively $P_1,\ldots,P_t$ be the orbits of the action on $X$, respectively $Y$.
Suppose that $R \subseteq X \times Y$ is a $G$-invariant relation and call $B = (b_{ij})$ the $n \times m$ matrix of this relation, i.e. $b_{ij}=1$ if and only if $x_i R y_j$ and $b_{ij} = 0$ otherwise.
\begin{itemize}
\item[(i)] The vectors $B^T f_{O_i}$, $i=1, \ldots, s$, are linear combinations of the vectors $f_{P_j}$.
\item[(ii)] If $B$ has full row rank, then $s \leq t$. If $s = t$, then all vectors $f_{P_j}$ are linear combinations
of the vectors $B^T f_{O_i}$, hence $f_{P_j} \in \im(B^T)$.
\end{itemize}
\end{lemma}

Under the conditions of Lemma \ref{lem:block}(ii), groups can be used to construct
low degree Boolean functions. In particular, it applies to Grassmann graphs
due to a result by Kantor \cite{Kantor1982}.
More generally, most of the examples
in \S\ref{sec:exc} were found by assuming a group action and investigating
the corresponding tactical decomposition with computational tools such as
integer linear programming (ILP) solvers and group theoretic methods.

\subsection{Erd\H{o}s-Ko-Rado Theorems \& Hoffman's Ratio Bound}\label{sec:ratio}

Let us briefly mention the Erd\H{o}s-Ko-Rado (EKR) theorem \cite{EKR}.
The EKR theorem for sets states that a family $Y$ of pairwise intersecting
elements of $\binom{[n]}{m}$, $2 \leq m \leq n/2$,
has size at most $\binom{n-1}{m-1}$. For $m < n/2$ equality occurs
if and only if $Y$ consists of all $m$-sets containing some fixed element.
Essentially the same holds for $\gauss{V}{m}$, $V = V(n, q)$.
This can be seen using Delsarte's LP bound or, more famously, the well-known
Hoffman's Ratio Bound which bounds the size
of a \textit{coclique} (or: {\it independent set}, {\it stable set}) of a graph.

\begin{theorem}[Hoffman's Ratio Bound]
 Let $A$ be the adjacency matrix of a graph $\Gamma$
 of order $v$ with degree $k$ and smallest eigenvalue $\lambda$.
 Suppose that $Y$ is a coclique of $\Gamma$.
 Then $|Y| \leq \frac{v}{1 - k/\lambda}$.
 If equality holds, then $f_Y \in \< {\bf j} \> + V(\lambda)$.
\end{theorem}

It has been observed by Lov\'asz \cite{Lovasz1979}
that this implies $|Y| \leq \binom{n-1}{m-1}$ for
an intersecting family in $J(n, m)$ (i.e., a coclique of $\Gamma_m$).
For $n > 2m$, equality implies that $f_Y \in V_0 + V_1$, so $\deg(f_Y) \leq 1$.
Similarly, Frankl and Wilson did show that
$|Y| \leq \gauss{n-1}{m-1}$ for an intersecting family in $J_q(n, m)$ \cite{FW1986}.
Similarly, for $n \geq 2m$, equality implies $f_Y \in V_0 + V_1$, that is, $\deg(f_Y) \leq 1$.
See \cite{GM2016} for a unified treatment.
In general, there has been a fruitful interaction between
the theory of intersecting families and low degree Boolean functions
going back to the work by Friedgut \cite{Friedgut2008}, and Dinur and Friedgut \cite{DF2009}.

\section{The Hypercube}

The most basic question for any graded poset under consideration is the classification
of the degree $1$ Boolean functions.\footnote{%
The term \textit{affine Boolean function} is maybe a less clumsy.
Here we avoid it due to the potential confusion with affine spaces.}
For the hypercube, this is trivial and one
sees that the only examples are $0$, $1$, $x_i$, and $1-x_i$: Say,
$f: \{ 0, 1\}^n \rightarrow \{ 0, 1 \}$ is not the constant function.
Write $f(x) = c+\sum c_i x_i$.
Clearly, $c, c_1, \ldots, c_n \in \{ 0, 1\}$ by considering the elements $x$
of $\{ 0, 1\}^n$ with weight at most $1$.
As we can always replace $f$ by $1-f$, assume that $f(0)=0$ or, equivalently,
$c=0$. By considering the elements of $\{ 0, 1\}^n$ of weight $2$, we see that
at most one $c_i$ can be equal to $1$.
In general, the spectral analysis on the hypercube is easy
as it possesses an orthogonal basis corresponding
to the well-known {\it Walsh-Hadamard transform}.

The approximate case is more interesting for degree $1$. It is known as the
Friedgut-Kalai-Naor (FKN) Theorem \cite{FKN2002}.
In line with the usual notation in Boolean function analysis,
we write $\|f\|_2^2 = \frac{1}{|X|} \sum_{x \in X} f(x)^2$.
We say that $f$ is {\it $\eps$-close} to $g$ if $\|f-g\|_2^2 < \eps$.
One possible formulation is as follows, see also \cite[\S2.6]{BFA}.

\begin{theorem}[FKN Theorem]
 Suppose that $f: \{ 0, 1\}^n \rightarrow \{ 0, 1 \}$ is $\eps$-close to $g$
 for some degree $1$ function $g: \{ 0, 1\}^n \rightarrow \RR$.
 Then $f$ is $O(\eps)$-close to $0$, $1$, $x_i$, or $1-x_i$ for some $i \in [n]$.
\end{theorem}

Here and in the following asymptotics with big-$O$-notation are always with respect to $n$.
For instance, read $\|f-g\|_2^2 = O(\eps)$ as
$\limsup_{n \rightarrow \infty} \|f-g\|_2^2 \leq C \eps$ for some constant $C$.
Note that $\|f-g\|_2^2 = \Pr(f \neq g)$ if $f$ and $g$ are Boolean.

This survey aims to state everything in the theory of association
schemes, so let us rephrase the FKN theorem accordingly.
This will give a better idea about how an ``FKN theorem for an association scheme''
should look like. Recall that $E_0$ and $E_1$ are idempotent matrices
which project onto $V_0 = \< j\>$ and $V_1$, respectively.

\begin{theorem}[FKN Theorem, alternative]\label{thm:fkn_alt}
 Suppose that $Y \subseteq \{ 0, 1 \}^n$
 has $\| (E_0+E_1) f_Y \|_2^2 \geq 1-\eps$.
 Then either $|Y|/2^n = O(\eps)$ or $1-|Y|/2^n = O(\eps)$ holds,
 or there exists an $i \in \{ 1, \ldots, n \}$
 such that
 $| Y \Delta \{ z \in \{ 0, 1\}^n: i \in z \} |/2^n = O(\eps)$
 or $| Y \Delta \{ z \in \{ 0, 1\}^n: i \notin z \} |/2^n = O(\eps)$.
\end{theorem}

A set $Y$ is called a \textit{$\ell$-junta} if $f_Y$ depends on at most
$\ell$ coordinates of its input. That is, there exists a set $L = \{ c_1, \ldots, c_\ell\} \subseteq [n]$
and a function $g\colon \{0, 1\}^\ell \rightarrow \{ 0, 1\}$
such that $f_Y(x_1, \ldots, x_n) = g(x_{c_1}, \ldots, x_{c_{\ell}})$.
Theorem \ref{thm:fkn_alt} states that an almost degree $1$ Boolean function
is close to a $1$-junta (also called a \textit{dictator}).
The term \textit{essential variable} is also used \cite{Tarannikov2000}.
Boolean degree $d$ functions have been first characterized as juntas by Nisan and Szegedy in 1994 \cite{NiS1994}.
The same result has been rediscovered by Tarannikov, Korolev, Botev in 2001 \cite[Theorem 6]{TKB2001}.

\begin{theorem}[Nisan-Szegedy Theorem]\label{thm:NS}
 Suppose that $Y \subseteq \{ 0, 1\}^n$
 has $\deg(f_Y) \allowbreak \leq d$.
 Then $f_Y$ depends on at most $d2^{d-1}$ coordinates, i.e., is a $d2^{d-1}$-junta.
\end{theorem}

Is this bound tight? For $d=2$ it is.
A complete classification of Boolean degree $2$ functions was to our
knowledge first given by Carlet et al. in \cite{CCCS1991}.
The following description is due to Yuval Filmus. We do not give complements:
\begin{align*}
  &0, ~~x, ~~x \text{ AND } y = xy, ~~x \text{ XOR } y = x+y-xy, ~~xy+(1-x)z, \\
  &\text{Ind}(x{=}y{=}z) = xy+xz+yx-x-y-z+1, \\
  & \text{Ind}(x {\leq} y {\leq} z {\leq} w \text{ OR } x {\geq} y {\geq} z {\geq} w).
\end{align*}
Here $\text{Ind}$ is the indicator function.

For $d=3$, the correct answer is $10$, not $12$ as suggested by Theorem \ref{thm:NS}.
The degree $3$ case has been classified by Kirienko (computationally) in \cite{TK2001}
and Zverev (humanly) in \cite{Zverev2007}. The extremal examples belong
to a general family of Boolean degree $d$ functions due to Tarannikov from 2000 \cite{Tarannikov2000}
which depends on $3 \cdot 2^{d-1}-2$ coordinates.
The same construction has been rediscovered in 2020 by Chiarelli, Hatami, and Saks \cite{CHS2020}.
Put
\[
 H_2(x_1, x_2, x_3, x_4) = x_1 \oplus x_2 \oplus (x_1 \oplus x_3)(x_2 \oplus x_4).
\]
This is an example for $d=2$. Now define $H_d$ recursively by
\[
 H_d(x_1, x_2, y, z) = x_1 \oplus (x_1 \oplus x_2 \oplus 1) H_{d-1}(y) \oplus (x_1 \oplus x_2) H_{d-1}(z).
\]
Tarannikov gives the construction in terms of \textit{resilient functions}.
This is equivalent with the notation of essential variables, see the remark after Proposition 6.23 in \cite{BFA}
or the introduction of \cite{KV2024}.
Another relevant terms is that of a \textit{correlation-immune function}.

For $d \geq 4$, exact values are unknown.

For the upper bound, a breakthrough occured with the recent result by Chiarelli, Hatami, and Saks
who did show an upper bound of $O(2^d)$ for the number of essential variables \cite{CHS2020}.
The current best upper bound is due to Wellens \cite{Wellens2022}.
We obtain the following state of the art:

\begin{theorem}\label{thm:nisan_szegedy_cur}
 Suppose that $Y \subseteq \{ 0, 1\}^n$
 has $\deg(f_Y) = d$ and $f_Y$ depends on the maximum number of coordinates (for that $d$).
 Then $f_Y$ depends on at least $3 \cdot 2^{d-1}-2$ and at most $8.788 \cdot 2^{d-1}$ coordinates.
\end{theorem}

Lastly, one can ask what happens when a function is close to degree $d$.
This question got answered by Kindler and Safra in \cite{KS2002}.

\subsection{The $p$-biased Hypercube}

The $p$-biased hypercube is a hypercube $\{ 0, 1 \}^n$ with a distribution
where each coordinate is $1$ with probability $p$.
Hence, for $x \in \{ 0, 1 \}^n$, we have a measure $\mu_p$
defined by $\mu_p(x) = p^{\sum x_i} (1-p)^{\sum (1-x_i)}$.
For $p=\frac12$ we have the classical hypercube.
Expanding Boolean function analysis to the $p$-biased setting
has been very fruitful in the last years, but it is outside the scope
of this document.
There are an FKN theorem \cite{Filmus2016,Filmus2021exposition}
and a Kindler-Safra-type theorem are known \cite{DFH2024}.
The utility of the $p$-biased hypercube is that it can give
a good model for more complex structures. For instance,
for $p = \min(\frac{m}{n}, 1-\frac{m}{n})$,
the $p$-biased hypercube and $J(n, m)$ behave similarly.

Recently, Tanaka and Tokushige suggested a $p$-biased measure
for $V(n, q)$ in the context of Erd\H{o}s-Ko-Rado type results
for vector spaces \cite{TT2024}.

\subsection{The General Hamming Scheme}

Consider the Hamming graph $H(n, q)$ with $q > 2$.
A FKN theorem due to Alon, Dinur, Friedgut and Sudakov can be found
in Lemma 2.4 in \cite{ADFS2004}.
One can show that Theorem \ref{thm:nisan_szegedy_cur}
implies that a Boolean degree $d$ function on $H(n, q)$
depends on at most $4.394 \cdot 2^{\lceil \log_2 q \rceil d}$ variables.
Recently, Valyuzhenich did show an upper bound $\frac{dq^{d+1}}{4(q-1)}$
for $q \neq 2^s$ \cite{Valyuzhenich2024}.


\section{The Johnson Scheme}\label{sec:johnson}

A Boolean degree $d$ function
of the hypercube is, restricted to words of weight $m$, also a Boolean
degree $d$ function of the Johnson scheme $J(n, m)$.
Any Boolean function on $J(n, m)$ has degree at most $\min(m, n-m)$.
Thus, one requires at least $\min(m, n-m) > d$ for a classification.
For $d=1$, this has been done many times.
For instance, it is a special case of more general results by Meyerowitz \cite{Meyerowitz1992}.
Short proofs limited to the degree $1$ case
can be found in \cite{Filmus2016,FI2019}.

\begin{theorem}
 Let $m, n-m \geq 2$.
 Suppose that $Y \subseteq \binom{[n]}{m}$
 has $\deg(f_Y) \leq 1$. Then $Y$ or its complement is one of the following:
 \begin{enumerate}
  \item the empty set (that is, $f = 0$),
  \item the set of all $m$-sets which contain a fixed element $i$ (that is, $f = x_i$).
 \end{enumerate}
\end{theorem}

Recall that the FKN theorem for the hypercube says that if a Boolean function is $\eps$-close to degree $1$,
then it is $O(\eps)$-close to a Boolean degree $1$ function. For the Johnson scheme,
we cannot hope for such a result. Say, for $m=2$ and $n$ large, $x_1 + x_2$
is close to degree $1$ on the Johnson scheme, while it is not on the hypercube.
Filmus gives an FKN Theorem for the Johnson scheme \cite{Filmus2016}.
Again, we will state the theorem in two formulations.

\begin{theorem}[FKN Theorem for the slice]
  Let $m, n-m \geq 2$ and put $p = \min(m/n, 1 - m/n)$.
  Suppose that $f: \binom{[n]}{m} \rightarrow \{ 0, 1\}$
  is $\eps$-close to $g$
  for some degree $1$ function $g: \binom{[n]}{m} \rightarrow \RR$.
  Then either $f$ or $1-f$ is $O(\eps)$-close to
  \[
   x_{i_1} + \cdots + x_{i_t} ~\text{ or }~ 1-(x_{i_1} + \cdots + x_{i_t})
  \]
  for some set $I = \{ i_1, \ldots, i_t \}$ of size at most $\max(1, O(\sqrt{\eps}/p))$.
\end{theorem}

It can be checked that the two functions in the statement are $O(\eps)$-close
to being Boolean. Thus, the theorem is tight. Let us also
state a version of the theorem using sets.

\begin{theorem}[FKN Theorem for the Johnson scheme]
 Let $m, n-m \geq 2$ and put $p = \min(m/n, 1 - m/n)$.
 Suppose that $Y \subseteq \binom{[n]}{m}$
 has $\| f_Y - (E_0+E_1) f_Y \|_2^2 \geq 1-\eps$.
 Then there exists an $I \subseteq \{ 1, \ldots, n \}$ with $|I| = \max(1, O(\sqrt{\eps}/p))$
 such that either $| Y \Delta \{ z \in \binom{[n]}{m}: |z \cap I| \geq 1 \} |/\binom{n}{m} = O(\eps)$
 or $| Y \Delta \{ z \in \binom{[n]}{m}: |z \cap I| = 0 \} |/\binom{n}{m} = O(\eps)$.
\end{theorem}

For $m=3$, equitable partitions of degree $2$ are classified in \cite{EGGV2022,GG2013}.

Let $\gamma(d)$ denote the maximum number of relevant coordinates for
a Boolean degree $d$ function on the hypercube. The main results of \cite{Filmus2023,FI2019a,FV2023}
show the following.

\begin{theorem}\label{thm:NS_slice}
 There exist a constant $C$ and functions $\xi, n_0: [d] \rightarrow \RR$
 such that the following holds.
 Suppose that $Y \subseteq \binom{[n]}{m}$ with $\deg(f_Y) \leq d$.
 \begin{enumerate}
  \item If $C^d \leq m \leq n - C^d$, then $f_Y$ depends on at most $\gamma(d)$ coordinates.
  \item If $n-m \geq n_0(d)$ and $m \geq 2d$, then $f_Y$ depends on at most $\gamma(d)$ coordinates.
  \item If $n \geq 2m$ and $m \geq 2d$, then $f_Y$ depends on at most $\xi(d)$ coordinates.
 \end{enumerate}
\end{theorem}
It is also shown in \cite{Filmus2023} that for every positive integer $m$, where $d \leq m \leq 2d-1$,
and any positive integer $\ell$, there exists a
$n \geq 2m$ and a $Y \subseteq \binom{[n]}{m}$
such that $f_Y$ depends on $\ell$ coordinates. To see this, take some integer $e$ such that
$\ell \leq de$. Consider the function $g$ defined by
\[
 g = \sum_{i=1}^e \prod_{j=1}^d x_{(d-1)i+j}.
\]
It is clear that $\deg(g) \leq d$
and that it depends on at least $de$ coordinates for $n \geq 2de$.
The condition $m \leq 2d-1$ guarantees that $g$ is Boolean.
Written differently, we have $g = f_Y$ for
\[
 Y = \bigcup_{i=1}^e \left\{ y \in \binom{[n]}{m}: \{ (d-1)i+1, \ldots, (d-1)i+d \} \subseteq y \right\}.
\]
In Theorem \ref{thm:NS_slice}(2), $\gamma(d)$ and $\xi(d)$ are not
the same. We have $\gamma(2) \leq 4$ by the Nisan-Szegedy theorem (and, indeed, equality),
but \cite[\S3]{DBDIM2023} gives an examples for $(n,m) = (8,4)$ which depends
on $5$ coordinates, so $\xi(2) \geq 5 > 4 = \gamma(2)$:
identify $J(8, 4)$ with the elements of $\{0,1\}^8$ with Hamming weight $4$
and take all elements which start with one of
\begin{align*}
  & 11000, 01100, 00110, 00011, 10001,\\
  & 11100, 01110, 00111, 10011, 11001.
\end{align*}
This construction generalizes to $J(2m, m)$ for all $m \geq 4$, see Construction 3 in \cite{Vorobev2020}.
A variant of the Kindler-Safra theorem for $J(n, m)$ due to Keller and Klein can be found in \cite{KK2020}.

In \cite{KMW2024} it is shown by Kiermaier, Mannaert and Wassermann that
\begin{align*}
 \mathrm{gcd}\left( \binom{n}{m}, \binom{n-1}{m-1}, \ldots, \binom{n-d}{m-d} \right)
\end{align*}
divides $|Y|$.

The spectral analysis on the Johnson scheme $J(n, m)$ is much helped
by the existence of a ``nice'' orthogonal basis. This basis has been
described by Srinivasan in \cite{Srinivasan2011} as well as by Filmus in \cite{Filmus2016a}.
See also \cite{FL2020}.

\section{The Grassmann Scheme}

In the Grassmann scheme $J_q(n, m)$,
Boolean degree $1$ functions are traditionally considered in the
projective space $\PG(n-1, q)$ and are called {\it Cameron-Liebler sets}
(or {\it Cameron-Liebler $(m-1)$-space classes} and {\it Cameron-Liebler line classes} for $m=2$).
Considering the very easy classification of Boolean degree $1$ functions
on the hypercube and in the Johnson scheme $J(n, m)$, it is tempting
to assume that for $n \geq 2m \geq 4$, the situation in the Grassmann
scheme is similar. That is, for $V = V(n, q)$, if $Y \subseteq \gauss{V}{m}_q$
with $\deg(f_Y) \leq 1$, then $Y$ is one of the obvious, say, {\it trivial examples}:
\begin{enumerate}[(I)]
 \item The empty set.
 \item All $m$-spaces through a fixed $1$-space $P$.
 \item All $m$-spaces in a fixed hyperplane $H$ of $V$.
 \item The union of the previous two examples when $P \nsubseteq H$.
 \item The complement of any of the examples (I) to (IV).
\end{enumerate}

Surprisingly, the classification question is one of the oldest ones in this survey
as it goes back to work by Cameron and Liebler from 1982 \cite{CL1982}.
Consider a subgroup $G$ of $\mathrm{P\Gamma{}L}(n, q)$ acting on the $1$-spaces
and $2$-spaces of $V(n, q)$. Then Lemma \ref{lem:block}
shows that $G$ has at least as many orbits on $2$-spaces as on $1$-spaces.
Examples (I)--(V) arise as orbits of such groups $G$.
Are there any more examples which arise as orbits in this way?
In 2008, Bamberg and Penttila did confirm the conjecture,
that is, there are no additional examples for such groups \cite{BP2008}.

Cameron and Liebler also relaxed their group theoretic
question about $G$ to a combinatorial question, limited to $J_q(4, 2)$.
Here one asks, in the language of this survey, about the Boolean degree $1$ functions
of $J_q(4, 2)$. They suggested that the list (I)--(V) is complete.
This turned out to be false. We will survey various constructions
in the next section.

A relevant portion of the literature on $J_q(4, 2)$ uses
projective notation and the Klein correspondence between
buildings of type $A_3$ and $D_3$. We refer the reader to Section 4.6 of
\cite{BeutelspacherRosenbaum1998} for a geometric point of view of the Klein correspondence
and to Chapter 12 of \cite{Taylor91} for an algebraic one.
The recent survey by Gavrilyuk and Zinoviev on completely regular
codes in $J(n, m)$ and $J_q(n, m)$ covers the degree $1$ case \cite{GZ2025}.

\subsection{Exceptional Degree $1$ Examples}\label{sec:exc}

Only in $J_q(4, 2)$ non-trivial example for
Boolean degree $1$ function/Cameron-Liebler sets are known.
As all the descriptions are in terms of projective geometry,
we will call $1$-spaces \textit{points} and $2$-spaces \textit{lines}.
There are $\gauss{4}{2} = (q^2+1)(q^2+q+1)$ lines.
Examples necessarily have size $x(q^2+q+1)$ for
some integer $x$, where $1 \leq x \leq q^2+1$, see Equation \eqref{eq:div_grass} in \S\ref{sec:VSdiv}.
By taking complements, we can always guarantee
$x \leq \frac{q^2+1}{2}$. There are no established
naming conventions for the various families, so we make some up
for the convenience of the reader.\footnote{%
If a construction can be clearly attributed to three or fewer authors,
then we name them after them. Otherwise, we tried to pick a descriptive name.
}
The group in the table is a known subgroup
of the the automorphism group of the example.
Usually, some modularity condition on $q$ is necessary.

{\medskip\footnotesize\noindent
\setlength{\tabcolsep}{5pt}
\begin{tabular}{@{\,}cc@{~~~}l@{~~}l@{~~}l}
$q$ & $x$  & group & name & ref \\
\hline
odd & $\frac{q^2+1}{2}$ & $O^-(4, q)$ & Bruen-Drudge & \S\ref{sec:bd_fam} \\
odd & $\frac{q^2+1}{2}$ & pt.-stab. in $O^-(4, q)$ & derived Bruen-Drudge & \S\ref{sec:derived_bd_fam} \\
odd & $\frac{q^2+1}{2}$ & $q^2(q+1)$ & derived Bruen-Drudge & \S\ref{sec:derived_bd_fam} \\
$1 \pmod{4}$ & $\frac{q^2+1}{2}$ & $PGL(2, q)$ & Cossidente-Pavese & \S\ref{sec:cp_fam} \\
$4$ & $7$ & stab.~of cone over hyperoval & Govaerts-Penttila & \S\ref{sec:gp_fam} \\
$5$ & $10$ & stab. of a sp.~cap & Gavrilyuk-Metsch & \S\ref{sec:gm_fam} \\
$5, 9 \pmod{12}$ & $\frac{q^2-1}{2}$ & $(C(q^2+q+1) \rtimes C(\frac{q-1}4))\rtimes C(3)$ & Affine I
& \S\ref{sec:r_one_fam} \\
$2 \pmod{3}$ & $\frac{(q+1)^2}{3}$ & $C(q^2+q+1) \rtimes C(3)$ & Affine II & \S\ref{sec:r_two_fam} \\
$27$ & $336$ & $(C(q^2+q+1) \rtimes C(2))\rtimes C(9)$ & Rodgers' Sporadic I & \S\ref{sec:r_spo_fam} \\
$32$ & $495$ & $C(q^2+q+1) \rtimes C(15)$ & Rodgers' Sporadic II & \S\ref{sec:r_spo_fam}
\end{tabular}\par}

\subsubsection{The Bruen-Drudge Family}\label{sec:bd_fam}

This family requires $q$ odd and has $x = \frac{q^2+1}{2}$.
It has been described by Drudge for $q=3$ \cite{Drudge1998}
and later in general by Bruen and Drudge \cite{BD1999}.
The stabilizer of the example is the finite simple group of type $O^+(4, q)$.

The example has a concise description: Recall that the finite
field $\GF(q)$ of odd order $q$ has $(q-1)/2$ non-zero squares
and $(q-1)/2$ nonsquares. Let $Q(x) = \xi x_1^2 + x_2^2 + x_3^2 + x^4$
where $\xi$ is some non-square. Then the discriminant $\disc(Q)$
of $Q$ equals $\xi$. It can be checked that there are $q^2+1$ points $\< x \>$
with $Q(x) = 0$, $q(q^2+1)/2$ with $Q(x)$ a non-zero square,
and $q(q^2+1)/2$ with $Q(x)$ a non-square.
For each subspace $S$, we can consider
$\disc_{|S}(Q)$, the discriminant of $Q$ restricted to $S$.
Then we have the following census of lines:
\begin{enumerate}[(i)]
 \item There $\binom{q^2+1}{2}$ lines $L$ with $\disc_{|L}(Q)$ a non-zero square.
 \item There are $\binom{q^2+1}{2}$ lines $L$ with $\disc_{|L}(Q)$ a non-square.
 \item There are $\frac{(q^2+1)(q+1)}{2}$ lines with $\disc_{|L}(Q) = 0$
       and $Q(x)$ either $0$ or a non-zero square for all points $\<x\>$ on $L$.
 \item There are $\frac{(q^2+1)(q+1)}{2}$ lines with $\disc_{|L}(Q) = 0$
       and $Q(x)$ either $0$ or a non-square for all points $\<x\>$ on $L$.
\end{enumerate}
Any union of size $\frac{q^2+1}{2} (q^2+q+1)$ gives an example, that is,
(i) and (iii), (i) and (iv), (ii) and (iii), (ii) and (iv).

\subsubsection{Derived Bruen-Drudge Families}\label{sec:derived_bd_fam}

There are several more families with $x = \frac{q^2+1}{2}$
which are obtained by modifying the Bruen-Drudge family locally.
All of them require $q$ odd.
One family has been described by Cossidente and Pavese \cite{CP2017}
and, independently, by Gavrilyuk, Matkin and Penttila \cite{GMP2016}
admitting the stabilizer of a point in $O^-(4, q)$,
one other family by Cossidente and Pavese \cite{CP2018}
admitting a subgroup of order $q^2(q+1)$.
Furthermore, Remark 3.16 in \cite{CP2018} notes that
there are $1, 3, 5$ additional examples for $q=7, 9, 11$, respectively.

\subsubsection{The Cossidente-Pavese Family}\label{sec:cp_fam}

One last family with $x = \frac{q^2+1}{2}$.
It requires $q \equiv 1 \pmod{4}$, admits $PGL(2, q)$ as a stabilizer,
and has been described by Cossidente and Pavese \cite{CP2019}.
Francesco Pavese also informed the author that there are
sporadic examples for $q=13, 25, 41$ known which admit
$C_{(q-1)/2} \times PGL(2, q)$ as group.

\subsubsection{The Govaerts-Penttila Example}\label{sec:gp_fam}

This example requires $q = 4$ and has $x = 7$.
It has been described by Govaerts and Penttila \cite{GP2005}.
This is the only known example for which $q$ is an even power of $2$.

It has a very concise description: Fix a projective plane $\pi$
(that is, a $3$-space). In $\pi$, take a hyperoval $\scrH$, that is,
a set of $q+2$ points with no three collinear. Pick a point $P$
not in $\pi$. Let $\scrC$ denote the cone with vertex $P$ and base $\scrH$.
Then the set consists of the following is an Boolean degree $1$ function with $x=7$:
\begin{enumerate}[(i)]
 \item The $q+2 = 6$ lines in $\scrC$.
 \item The $\frac{(q+2)(q-1) \cdot (q+1)(q-1)}{2} = 135$ lines which meet $\scrC$
        in precisely two points and are disjoint from $\scrH$.
 \item The $\frac{q(q-1)}{2} = 6$ lines in $V(4, q)$ which are disjoint from $\scrH$.
\end{enumerate}

\subsubsection{The Gavrilyuk-Metsch Example}\label{sec:gm_fam}

This example requires $q = 5$ and has $x = 10$.
It is described by Gavrilyuk and Metsch \cite{GMe2014}.
It uses a cap of $20$ points which has been constructed in \cite{AKL1996}.
For $q=5$, this is the unique example with $x = 10$.

\subsubsection{Affine Family I} \label{sec:r_one_fam}

This family requires $q \equiv 5 \pmod{12}$ or $q \equiv 9 \pmod{12}$
and has $x = \frac{q^2-1}{2}$. Small examples for this family have been
first discovered by Rodgers in his PhD thesis \cite{PhDRodgers}. Descriptions
of the infinite family are due to De Beule, Demeyer, Metsch and Rodgers \cite{DeBeule2016}
as well as Feng, Momihara and Xiang \cite{FMX2015}.

We call a family {\it affine} if it has the property
that all their elements are not contained in a fixed (hyper)plane $\pi$.
In particular, for $x = \frac{q^2-1}{2}$, we can add all the
lines in $\pi$ to obtain a family with $x = \frac{q^2+1}{2}$.
This also makes them non-trivial examples for
Cameron-Liebler sets in affine spaces, see \cite{DIMS2021,DMSS2020} and \S\ref{sec:bilin}.

For $q \equiv 9 \pmod{12}$, the family belongs to an affine two-intersection set
with parameters $(\frac12 (3^{2e} - 3^e), \frac12 (3^{2e} + 3^e))$.
That is, there exists a set of points $\scrP$ not in $\pi$
such that any line not in $\pi$ intersects $\scrP$ in one of these two numbers.
The one type of lines corresponds to the lines in the family.

\subsubsection{Affine Family II} \label{sec:r_two_fam}

This family requires $q \equiv 2 \pmod{3}$ and has $x = \frac{(q+1)^2}{3}$.
Small examples for this family have been
first discovered by Rodgers \cite{PhDRodgers}.
The infinite family has been described \cite{FMRXZ2021}
by Feng, Momihara, Rodgers, Xiang and Zou.
The family includes the only infinite family for $q$
an odd power of $2$.

\subsubsection{Rodgers' Sporadic Affine Examples} \label{sec:r_spo_fam}

In \S4.3 of \cite{PhDRodgers}, Rodgers describes
one sporadic example with $(q, x) = (27, 336)$
and two sporadic examples with $(q, x) = (32, 495)$.

\subsection{Classification Results for Degree $1$}

We have seen that in $J_q(4, 2)$ a classification of Boolean degree $1$ functions
is hopeless. The general question of the classification of Boolean degree $1$ function for $J_q(n, m)$
is surely (at least for $m=2$) implicit in the work by Cameron and Liebler from 1982.
The case for general $n$ and $m=2$ has been intensively investigated
by Drudge in his PhD thesis \cite{Drudge1998} in 1998, while the case for general $m$
found a more explicit treatment in the late 2010s \cite{BDBD2019,FI2019,RSV2018}.

The crucial ingredient for several of the classification results
is the use of design-orthogonality together with
the following weighted design in $J_q(4, 2)$ due to Gavrilyuk and Mogilnykh \cite{GM2014}.
For a point $P$ and a hyperplane $H$, define $g_{P,H}: \gauss{V(4, q)}{2} \rightarrow \{ -(q-1), 0, 1\}$ by
\begin{align*}
 g_{P,H}(L) = \begin{cases}
            1 & \text{ if } P \subseteq L \nsubseteq H,\\
            1 & \text{ if } P \nsubseteq L \subseteq H,\\
            -(q-1) & \text{ if } P \subseteq L \subseteq H,\\
            0 & \text{ otherwise.}
           \end{cases}
\end{align*}
One verifies that $g_{P,H}^T {\bf j} = q^2+1$.
A line $L$ contains $q+1$ points $P_1, \ldots, P_{q+1}$
and lies on $q+1$ planes $H_1, \ldots, H_{q+1}$.
By Equation \eqref{eq:designorth}, $f_Y^T g_{P_i,H_j} = x$ for a Boolean degree $1$ function $f_Y$
of size $x(q^2+q+1)$. Hence,
\begin{align*}
 x = & |\{ L \in Y: P \subseteq L \}| + |\{ L \in Y: L \subseteq H \}| \\ &- (q+1) |\{ L \in Y: P \subseteq L \subseteq H\}|.
\end{align*}
For a fixed line $L$, let $t_{ij}$ denote the number of elements of $Y$ in $P_i$
and $H_j$ (distinct from $L$).
Together with the fact that $\{ Y, X \setminus Y\}$ is an equitable bipartition,
the possible values of $t_{ij}$ (called pattern) are very restricted, see \cite{GMe2014}.

The combined work by Drudge \cite{Drudge1998},
by Gavrilyuk, Matkin and Mogilnykh \cite{GM2018,GM2014,Matkin2018}, and by Filmus and the author \cite{FI2019}
gives a complete classification result for small $q$:

\begin{theorem}\label{thm:class_smallq}
 Let $q \in \{ 2, 3, 4, 5\}$.
 Suppose that $Y$ is a family of $m$-spaces of $V(n, q)$
 with $\deg(f_Y) \leq 1$.
 If $Y$ is not trivial, then $(n,m) = (4, 2)$.
\end{theorem}

This result is based on a complete classification of the non-trivial Boolean degree $1$ functions
in $J_q(4, 2)$ and then showing that these do not extend to $J_q(5, 2)$.
Recently, the author showed in \cite{Ihringer2024} that for $|n-2m|$ sufficiently large,
the same classification results holds.

\begin{theorem}\label{thm:class_asymp}
 Let $\min(n-m, m) \geq 2$. Then there exists a function $c(q)$
 such that the following holds:
 Suppose that $Y$ is a family of $m$-spaces of $V(n, q)$
 with $\deg(f_Y) \leq 1$ and $|n-2m| \geq c(q)$.
 Then $Y$ is trivial.
\end{theorem}

The constant $c(q)$ depends on vector space Ramsey numbers (which exist, see \cite{GLR1972}).
For most parameters these implicitly depends on the repeated application of the
Hales-Jewett theorem. Thus, the $c(q)$ given by the proof is very large.
By Theorem \ref{thm:class_smallq}, we know $c(2) = 0$ and $c(3) = c(4) = c(5) = 1$.
For small $n$, there is a plenitude of results which restrict the
possible sizes of non-trivial examples.

Let us now summarize what is known for small $n$.
The best investigated case is $J_q(4, 2)$.
Recall that here any $Y$ with $\deg(f_Y) \leq 1$ has $|Y| = x(q^2+q+1)$ for some
integer $x$. If $Y$ is one of the trivial examples, then
$x \in \{ 0, 1, 2, q^2-1, q^2, q^2+1\}$
and we can without loss of generality assume that $x \leq \frac{q^2+1}{2}$
(as we can always consider the complement of $Y$).
Here results by Metsch \cite{Metsch2010,Metsch2014} and Metsch and Gavrilyuk \cite{GMe2014}
show the that non-trivial examples are very restricted in terms of their
possible sizes. We summarize these in the following statement.
\begin{theorem}\label{thm:xlines}
 Suppose that $Y$ is a family of $2$-spaces of $V(4, q)$
 of size $x(q^2+q+1)$ with $\deg(f_Y) \leq 1$.
 If $2 < x$, then the following holds:
 \begin{enumerate}
  \item We have $x \geq q+1$.
  \item We have $x > q \sqrt[3]{q/2} - \frac23 q$.
  \item The equation
  \begin{align*}
   \binom{x}{2} + \ell(\ell-x) \equiv 0 \pmod{q+1}
  \end{align*}
  has an integer solution in $\ell$.
 \end{enumerate}
\end{theorem}

In the general case, the first bounds on $x$ were found by Rodgers, Storme, and Vansweevelt \cite{RSV2018} for $n=2m$
and by Blokhuis, De Boeck, and D'haeseleer \cite{BDBD2019} for general $n$.
The currently best known conditions on the sizes, which the author
is aware of, are summarized in the following theorem:

\begin{theorem}
 Suppose that $Y$ is a family of $m$-spaces of $V(n, q)$
 of size $x\gauss{n-1}{m-1}$ with $\deg(f_Y) \leq 1$.
 If $Y$ is non-trivial, then the following holds:
 \begin{enumerate}
  \item If $(n, m) = (6, 3)$, then $x \geq \frac13 q$.
  \item If $n = 2m$ and $m \geq 3$, then $x > \sqrt[3]{q/2}$.
  \item If $n = 2m$, $m \geq 3$, and $q \geq q_0$ for some universal constant $q_0$, then $x \geq \frac15 q$.
  \item If $n \geq 3m-1$ and $m \geq 2$, then $x > \frac18 q^{n-3m+2}$.
  \item If $n \geq 3m+1$ and $m \geq 2$, then
  \[
    x \geq (\frac{1}{\sqrt[8]{2}} + o(1)) q^{n - \frac52 m + \frac12}.
  \]
 \end{enumerate}
\end{theorem}
\begin{proof}
 The first claim is Theorem 1.3 in \cite{Metsch2017}.
 The second claim is Theorem 6.7 in \cite{RSV2018}.
 The third claim is Theorem 1.4 in \cite{Metsch2017}
 (with the minor improvement from Theorem 1.8 in \cite{Ihringer2019}).
 The fourth claim is Theorem 5.1 in \cite{Ihringer2024}.
 The fifth claim is Theorem 3.6 in \cite{DBMS2024}.
\end{proof}

There are also general modular conditions which can be found
in \cite{DBM2022,DBMS2024}.
Note that these are less strong than
the results for $(n, m) = (4, 2)$.

A particular special case are {\it two-intersection sets}.
A two-intersection set is a family of $1$-spaces $\scrP$
such that any $m$-space contains either precisely $\alpha$
or precisely $\beta$ elements of $\scrP$.
There are many examples for such sets for $(n,m) = (3,2)$,
but already for $(n,m) = (4, 2)$ only trivial constructions
are known: $\scrP$ is the empty set ($\alpha=\beta=0$), all the $1$-spaces ($\alpha=\beta=q+1$),
a hyperplane ($\alpha = 1, \beta = q+1$), or the complement of a hyperplane ($\alpha=0, \beta=q$).
If there is a non-trivial example for $n \geq 4$ and $m=2$,
then $q$ must be an odd square. The first open case is $(n, m, q) = (4, 2, 9)$.
The work by Tallini Scafati from the 1970s discusses this in detail \cite{TS1972,TS1976}.\footnote{%
In particular, the problem of the existence of a two-intersecting with respect to lines in $\PG(3, 9)$
has been open for more than five decades. This deserves a footnote.}
There is a close connection to so-called {\it two-weight codes}, see Theorem 12.5 in \cite{CK1986}.

If $Y$ is the set of all $m$-spaces which contain precisely $\alpha$
elements of $\scrP$, then $\deg(f_Y) \leq 1$.
Here Theorem \ref{thm:class_asymp} implies that there are no
non-trivial examples if $|n-2m| \geq c(q)$. See \cite[\S4]{Ihringer2024}.

\subsection{Results Beyond Degree $1$}

The results beyond the degree $1$ case are sparse.

\subsubsection{The Khot-Minzer-Safra Theorem}

Let us start with the most prominent result,
the recent proof of the 2-to-2 Games Conjecture by Khot, Minzer, and Safra \cite{KMS2023}.
Loosely speaking, the Unique Games Conjecture states that it is hard
to approximate a certain type of NP-complete problem and is an important
conjecture in complexity theory. The 2-to-2 Games Conjecture is a
weaker version and its proof has been a great breakthrough in theoretical computer science.
We refer to Boaz Barak's blog post \cite{Barak2018} for an excellent summary.

Put $V = V(n, 2)$ and $X = \gauss{V}{m}$.
For $Y \subseteq X$, an $r$-space $R$ and an $s$-space $S$,
let $Y(R, S)$ denote the set of all elements $y$ of $Y$ with
$R \subseteq y \subseteq S$.
Informally, the main result by Khot, Minzer, and Safra is as follows:
If $Y$ has significant weight on low degree,
then there exist $R, S$ such that $Y(R, S)$ is a significant proportion
of $X(R, S)$.
Formally, define $\Phi(Y)$ by $\Phi(Y) = \frac{|E(Y,\overline{Y})|}{k \cdot |Y|}$,
where $k$ is the degree of $J_q(n, m)$.

\begin{theorem}[Khot-Minzer-Safra]
 For all $\alpha \in (0, 1)$ there exists a $\eps > 0$ and an integer $t \geq 0$
 such that for all $m \geq m_0(\alpha)$ and all $n \geq n_0(m, \alpha)$,
 the following holds. Suppose that $Y$ is a family of $m$-spaces in $V(n, 2)$
 with $\Phi(Y) \leq \alpha$.
 Then there exist subspace $R, S$ with $\dim(R) + \codim(S) \leq t$ and
 \[
  \frac{|Y(R, S)|}{|X(R, S)|} \geq \eps.
 \]
\end{theorem}

Recall that the distinct eigenvalues $\theta_j$ of $J_2(n, m)$ are:
\[
 \theta_j = P_{j1} = 2^{j+1}(2^{m-j}-1)(2^{n-m-j}-1) - (2^j-1) \approx 2^{n-j} \text{ for } j \ll n.
\]
Using these eigenvalues in Equation \eqref{eq:projection} we see that $\Phi(Y) \leq \alpha$ for some constant $\alpha$
is equivalent with $\left\| \sum_{j=0}^{d(\alpha)} E_j f_Y \right\|/\|f_Y\| \geq \beta(\alpha)$.
For instance, for $\alpha = \frac12$, we can take $d(\alpha) = 1$ and $\beta(\alpha) = 1 + o(1)$.
Hence, their result does indeed describe the structure of families $Y$ with a
significant weight on low degree.


\subsubsection{Degree $2$}

The degree $2$ in $J_q(n, m)$ has been investigated
for small dimension $n$ in \cite{DBDIM2023}.
In contrast to Theorem \ref{thm:NS_slice}(2),
but similar to the behavior for degree $1$ in $J_q(4, 2)$,
$n-m,m \geq 2d$ does not seem to suffice for an FKN-type result.
The authors provide the
following example in Section 5.2 for $(n,m) = (8,4)$ which is reminiscent
of the Bruen-Drudge example for degree $1$.
The construction goes as follows.
Define a quadratic form $Q$ by
\[
 Q(x) = x_1^2 + \alpha x_1x_2 + \beta x_2^2 + x_3^2 + x_4^2 + x_5^2 + x_6^2 + x_7^2 + x_8^2
\]
such that $1 + \alpha x + \beta x^2$ is irreducible over $\GF(q)$.
There are six types of $4$-spaces with respect to $Q$.
Take for $Y$ a union of the three types with a radical of dimension
at least $2$. Then $\deg(f_Y) = 2$.

Many non-trivial degree $2$ examples for $(n,m) = (6,3)$
can be found in \cite{DBDIM2023}. As for $J(n, m)$,
a Nisan-Szegedy type theorem (in a narrow sense)
in $J_q(n, m)$ is impossible for $m < 2d$.

\subsubsection{Divisibility Conditions}\label{sec:VSdiv}

Kiermaier, Mannaert and Wassermann recently investigated
divisibility conditions for the general case.
Let $Y$ be a family of $m$-spaces of $V(n, q)$
such that $\deg(f_Y) \leq d$. Then Theorem 4.7
in \cite{KMW2024} shows that
\begin{align}
 \gcd\left( \gauss{n}{m}, \gauss{n-1}{m-1}, \ldots, \gauss{n-d}{m-d} \right) \label{eq:div_grass}
\end{align}
divides $|Y|$. For $n$ even and $d=2$, this implies that
$|Y| = x\gauss{n-1}{m-1}$ for some integer $x$.
For instance, Lemma 4.11 in \cite{KMW2024} shows that $|Y| = x \gauss{n-d}{m-d}$ for
some integer $x$ if $m-i \,|\, n-i$ for all $i \in \{ 0, \ldots, d-1 \}$.

\section{Other Structures}

Let us mention some further results in some other structures.

\subsection{Permutation Groups}

Affine Boolean functions in the symmetric group $Sym(n)$ have been
classified in \cite{EFP2011}. FKN theorems for $Sym(n)$
can be found in \cite{EFF2015,EFF2015a,Filmus2021}. While
a Nisan-Szegedy type theorem does not hold in the sense
that there is a junta, it has been shown in \cite{DFLLV2020}
that a {\it constant depth decision tree} suffices to decide the value of
Boolean degree $d$ function on $Sym(n)$. An approximate version
of this result holds too. The different {\it complexity measures}
discussed in \cite{DFLLV2020} also give inspiration for how
classification results might need to be phrased.

The more general setting of Cameron-Liebler sets (or
Boolean degree $1$ functions) of permutation groups has been recently investigated in \cite{DMP2024}.

\subsection{Bilinear Forms, Affine Spaces, Polar Spaces} \label{sec:bilin}

Bilinear forms \cite[\S9.5]{BCN}, affine spaces,
and polar spaces \cite[\S9.4]{BCN} all share the property that
they are closely related to Grassmannians.

\subsubsection{Bilinear Forms}

The bilinear forms graph $\Gamma$ can be seen as an induced subgraph
of the Grassmann graph $J_q(n, m)$: Fix a $(n-m)$-space $S$.
The induced subgraph of $J_q(n, m)$ on all $m$-spaces disjoint
from $S$ is the bilinear forms graph, see \cite[\S9.5]{BCN}.
Barak, Kothari, and Steuer give a proof of the $2$-to-$2$
games conjecture using $\Gamma$ with $q=2$
instead of $J_2(n, m)$ \cite{BKS2019}.
In \cite{FI2019} it is shown that the restriction of the
Bruen-Drudge example (see \ref{sec:bd_fam}) to $\Gamma$ is also non-trivial (in some sense).
In \cite{BCGG2018} equitable partitions for $(q, m) = (2, 2)$ are studied.
Bounds on the possible sizes of degree $1$ examples are
given in \cite{Guo2021,Ihringer2021}.
One central important property of the hypercube and the Johnson graph,
which helps with the investigation of low degree Boolean functions,
is \textit{hypercontractivity}. Unlike for the Grassmann scheme,
hypercontractivity results exists for the bilinear forms scheme, see \cite{EKL2022,EKL2024}.

\subsubsection{Affine Spaces}

The affine case of Cameron-Liebler sets/Boolean degree $1$ functions
has been investigated in \cite{DBM2022,DBMS2022,DIMS2021,DMSS2020} and an asymptotic classification
is implied by \cite{Ihringer2024}. Classification results in the affine setting are strictly
easier as any Boolean degree $d$ function on the affine $V(n-1, q)$ is
a Boolean degree $d$ function on $V(n, q)$.
The modular condition in Theorem \ref{thm:xlines} generalizes to affine spaces
for $m=2$ \cite{DBM2022}. Here it reads
\[
 x(x-1) \frac{q^{n-3}-1}{q-1} \equiv 0 \pmod{2(q+1)},
\]
where $x$ only depends on the size of $Y$.
The affine case also has applications in property testing. For instance,
Theorem 2.4 in \cite{KM2022} is a weak type of FKN theorem
which is used for testing Reed-Muller codes.
A similar application for degree $d$ functions can be found in Theorem 1.1 in \cite{MZ2023}.

Note that there are two natural ways of defining degree in the affine lattice
by considering the dual.
We are not aware of any investigations of low degree Boolean functions for the dual.

\subsubsection{Polar Spaces}

Let $\sigma$ be a non-degenerate, reflexive sesquilinear form on $V(n, q)$.
A \textit{classical polar space}
consists of the subspaces of $V(n, q)$ which vanish
on $\sigma$, that is, the isotropic subspaces with respect to $\sigma$.
For $m \geq 2$, this gives us all the classification
problems as for bilinear forms and affine spaces on an induced subgraph of $J_q(n, m)$.
See \cite{DBD2020} for an investigation for maximal isotropic subspaces
and \cite{FI2019} for $2 \leq m \ll n$.
Note that at least in the case of maximal isotropic subspaces
the term {\it Cameron-Liebler class} does not correspond to
a degree $1$ function in this context: the former is defined
using design-orthogonality with so-called {\it spreads}, while
degree $1$ is defined as in this survey.

As for affine spaces, there are two choices for the degree:
Either $1$-spaces or maximal isotropic subspaces have rank $1$.
The above discusses the former.
For the latter and $m=1$, the corresponding families are called \textit{tight sets}.
FKN-type and Nisan-Szegedy-type theorems are surely a very challenging problem
due to the amount of examples known.
See \cite{BKLP2007} as starting point.
An up-to date survey can be found in Chapter 2 in \cite{BvM}.

\paragraph*{Acknowledgements} The author thanks Yuval Filmus,
Alexander Gavrilyuk, Jonathan Mannaert, Dor Minzer, Francesco Pavese,
Morgan Rodgers, and Yuriy Tarannikov for discussing topics related to this
survey. We also thank the anonymous reviewer for carefully reading this manuscript.

{\small

\bibliographystyle{plain}

}

\end{document}